\documentclass[10pt]{amsart}
\usepackage{palatino, mathpazo, amsmath, amsfonts, amssymb, mathrsfs, latexsym}
\usepackage[mathscr]{eucal}
\usepackage[all]{xy}
\usepackage{amssymb}
\allowdisplaybreaks[3]

\newtheorem{theorem}{Theorem}[section]
\newtheorem{proposition}[theorem]{Proposition}

\newtheorem{corollary}[theorem]{Corollary}

\numberwithin{equation}{section}

\theoremstyle{remark}
\newtheorem{remark}[theorem]{Remark}
\newtheorem*{remark*}{Remark}

\newtheorem{assumption}[theorem]{Assumption}

\newcommand{\ZZ}{{\mathbb Z}}
\newcommand{\RR}{{\mathbb R}}
\newcommand{\CC}{{\mathbb C}}
\newcommand{\VV}{{\mathbb V}}
\newcommand{\QQ}{{\mathbb Q}}
\newcommand{\Hom}{{\mathbb Hom}}
\newcommand{\LL}{{\mathbb L}}

\newcommand{\sO}{\mathscr{O}}
\newcommand{\cA}{\mathcal{A}}
\newcommand{\bP}{\mathbb{P}}

\newcommand{\bN}{\mathbf{N}}
\newcommand{\Cstar}{\mathbb{C}^{\times}}

\newcommand{\tD}{\widetilde{D}} 
\newcommand{\tomega}{{\tilde{\omega}}} 
\newcommand{\tX}{\widetilde{X}} 
\newcommand{\on}{\operatorname}

\DeclareMathOperator{\ex}{ex}
\DeclareMathOperator{\Cone}{Cone}
\DeclareMathOperator{\RHom}{RHom}
\DeclareMathOperator{\eval}{eval}

\title[Bondal-Orlov functors for toric DM stacks]
{Note on the Bondal-Orlov Functors for Toric DM Stacks}

\author{Yunfeng~~Jiang}
\address{Department of Mathematics\\ University of Kansas\\ 405 Snow Hall, 1460 Jayhawk Blvd\\ Lawrence, KS 66045 \\ USA}
\email{y.jiang@ku.edu}

\begin{document}

\maketitle

\begin{abstract}
We calculate explicit formulas for the general equivariant Bondal-Orlov  functors on the localized K-theory groups 
for a crepant birational transformation of toric DM stacks.  We recall some facts that the Bondal-Orlov functors give  equivalences on the bounded derived categories.  Applying twice of these functors we get the Seidel-Thomas spherical twists  for the derived category. 
\end{abstract}


\setcounter{section}{0}

\section{Introduction}

In this short note we calculate explicit formulas for the general equivariant Bondal-Orlov functors 
for  a crepant birational transformation of toric Deligne--Mumford (DM) stacks. 

Toric DM stacks were introduced by 
Borisov--Chen--Smith \cite{BCS} using stacky fans.  
The notion of extended stacky fan was introduced by Jiang in \cite{Jiang}, and it 
turns out that there is a one-to-one correspondence between the extended 
stacky fans and GIT data construction of toric DM stacks.
Given GIT data determined by a stability parameter $\omega$, 
we denote the toric DM stack by $X_\omega$,  
whose construction is reviewed in \S~\ref{sec:GIT}.
More details can be found in \cite{CIJ}. 
Birational transformation of toric DM stacks can be understood as changing the GIT stability parameters in the space of GIT stability conditions.

We study a special case of birational transformation of toric DM stacks: the \emph{crepant} birational transformations.  
We consider a special class of crepant birational transformations
($K$-equivalences) of toric DM stacks by a single wall crossing.  
The construction of such wall crossing can be found in \cite[\S~5.1]{CIJ}.  
There is a big torus $T$ action on the toric DM stack $X_\omega$,
and we work on the $T$-equivariant $K$-theory and bounded derived category on $X_\omega$. 
Y.~Kawamata in \cite{Ka} proves that a natural Fourier--Mukai transform 
induces equivalences of the bounded derived categories of 
$K$-equivalent toric DM stacks.
It was shown in \cite{CIJS} that the $T$-equivariant derived categories are also equivalent. 
In \cite[\S~6]{CIJ}, the authors calculated the equivariant Fourier--Mukai transform for $K$-theory basis of $X_\omega$ when restricted to torus fixed points. 

In this paper we  calculate explicit formulas for the general Bondal-Orlov functors in terms of  equivariant $K$-theory basis 
for a single toric wall crossing.  Let 
$$\varphi: X_+:=X_{\omega_{+}}\dasharrow X_-:=X_{\omega_{-}}$$ 
be a crepant transformation by a single wall crossing corresponding to the stability conditions $\omega_+$ and $\omega_-$.   The $T$-equivariant $K$-theory $K_0^T(X_\pm)$
are generated by equivariant line bundles corresponding to the lattice in the secondary fan. 
There is a common blow-up $\tX$ for both $X_+$ and $X_-$ and two contract maps $f_\pm: \tX\to X_\pm$.  
Let $E\subset \tX$ be the exceptional divisor. 
The general Bondal-Orlov functors are defined by:
\begin{equation}\label{BOk}
\mathbb{BO}_k=(f_+)_\star(\sO_{\tX}(kE)\otimes (f_-)^\star(--)): D_{T}^b(X_-)\to D_{T}^b(X_+)
\end{equation}
for any integer $k\in\ZZ$.  
We prove that $\mathbb{BO}_k$ is an equivalence on the equivariant bounded derived categories for any $k$. 
When $k=0$, $\mathbb{BO}_0$ is the usual Fourier-Mukai transform 
$\mathbb{FM}=(f_+)_\star((f_-)^\star(--))$. So $\mathbb{BO}_k$ can be taken as  {\em generalized} Fourier--Mukai transforms. 
The functors $\mathbb{BO}_k$, of course,  induce isomorphisms on the equivariant $K$-theory groups. 
Our computation gives explicit formulas of the Bobdal-Orlov functors
$\mathbb{BO}_k$ on the localized $K$-theory basis. See Theorem~\ref{theorem:GFM}.  
This generalizes the calculation of Theorem 6.19 in \cite{CIJ} for the  Fourier-Mukai transform $\mathbb{BO}_0=\mathbb{FM}$, although the proof is basically the same as in \cite{CIJ}.  In Theorem 6.23 of \cite{CIJ}, the authors prove that the Fourier-Mukai transform $\mathbb{BO}_0$ matches the analytic continuation of the $H$-functions for $X_\pm$, which implies the invariance of big quantum cohomology of $X_\pm$, see \cite[\S 5, 6]{CIJ} for details.  It is pretty interesting if the general Bondal-Orlov functors $\mathbb{BO}_k$ can match the analytic continuation of some hypergeometric functions for $X_\pm$.   

We also recall the fact that the Bondal-Orlov functors give an equivalence on the bounded derived categories of a single toric wall crossing.  The proof is based on the method of window shifted functor for the derived categories under GIT quotients by \cite{DS}, \cite{ADM} and \cite{HLS}.  We completely follow the proof of \S 5 in \cite{CIJS}.  Applying back for the Bondal-Orlov functor we get an autoequivalence of the bounded derived category which is called the spherical twist functor associated with a line bundle on the contraction locus in the sense of Seidel-Thomas in \cite{STh}.
We also give a proof that for a crepant birational transformation of toric DM stacks via a single wall crossing, the contraction locus are always weighted projective stacks.  This result is hidden somewhere in \cite{CIJ}, but there is no explicit explanation. 
The result presented here  is related to the monodromy conjecture in \cite{BMO} for Gromov-Witten theory of symplectic smooth DM stacks, see \cite{JT2}.
The result of the spherical twists can also be applied to find a correspondence for the Chen-Ruan cohomology for quasi-simple orbifold flops, see \cite{CJL}. 

This short note is organized as follows. In \S \ref{sec:GIT-CT} we review the construction of the cerpant transformation of toric DM stacks by a single wall crossing.  We calculate the general equivariant Bondal-Orlov functor on the localized $K$-theory basis  for the wall crossing of toric DM stacks in \S \ref{BO}.  In \S \ref{derived:equivalence} we recall the fact that the general equivariant Bondal-Orlov functors give an equivalence on the bounded derived categories for the wall crossing of toric DM stacks, and relate them to spherical twist associated with line bundles on the contraction locus.

\subsection*{Acknowledgement} 
Y. J. would like to thank E. Segal for valuable discussions on the Bondal-Orlov functors and spherical twists for toric DM stacks. 
Y. J.\ is partially supported by Simons Foundation Collaboration Grant 311837, and NSF Grant DMS-1600997.

\section{Crepant transformation of toric DM stacks}\label{sec:GIT-CT}
In this section  we review some basic facts and establish notations.
The main reference is \cite{CIJ}.

\subsection{Toric Deligne--Mumford stack and GIT quotient}\label{sec:GIT}

An \emph{$S$-extended stacky fan} 
is a quadruple $\mathbf{\Sigma}= (\bN,\Sigma,\beta,S)$, where:
\begin{itemize} 
\item $\bN$ is a finitely generated abelian group (torsions allowed); 
\item $\Sigma$ is a rational simplicial fan in $\bN\otimes \RR$;   
\item $\beta \colon \ZZ^m \to \bN$ is a homomorphism; 
we write $b_i = \beta(e_i)\in \bN$ for the image of the $i$th standard 
basis vector $e_i\in\ZZ^m$,
and write $\overline{b}_i$ for the image of $b_i$ in $\bN\otimes \RR$; 
\item $S \subset \{1,\dots,m\}$ is a subset, 
\end{itemize} 
such that:
\begin{itemize} 
\item each one-dimensional cone of $\Sigma$
is spanned by $\overline{b}_i$ for a unique 
$i\in \{1,\dots,m\}\setminus S$, and
each $\overline{b}_i$ with 
$i\in \{1,\dots,m\} \setminus S$ spans a one-dimensional 
cone of $\Sigma$; 

\item for $i\in S$, $\overline{b}_i$ lies in the support $|\Sigma|$ 
of the fan. 
\end{itemize} 
The vectors $b_i$ for $i\in S$ are called \emph{extended vectors}. 

The \emph{toric DM stack} associated to an extended stacky fan $(\bN,\Sigma,\beta,S)$ depends only on the underlying stacky fan and is defined as
the quotient stack
\[ 
 X_{\mathbf{\Sigma}} := [ U/K ], \quad \text{with} \quad 
 U=\CC^m\setminus \VV(I_{\Sigma}),
\]
where 
$I_{\Sigma}$ is the irrelevant ideal of the fan and $K:=\Hom(\LL^\vee, \Cstar)$ acts on $\CC^m$ through the data of extended stacky fan.

We require that the extended stacky fans $(\bN, \Sigma, \beta,S)$
satisfy the following conditions: 
\begin{itemize} 
\item[(C1)] the support $|\Sigma|$ 
of the fan is convex and full-dimensional; 
\item[(C2)] there is a strictly convex piecewise-linear function 
$f\colon |\Sigma| 
\to \RR$ that is linear on each cone of $\Sigma$; 
\item[(C3)] the map $\beta \colon \ZZ^m \to \bN$ is surjective. 
\end{itemize} 
The first two conditions are geometric constraints on $X_{\mathbf{\Sigma}}$: 
they are equivalent to saying 
that the corresponding toric stack $X_{\mathbf{\Sigma}}$ 
is semi-projective and has a torus fixed point. 
The third condition can be always achieved by adding  
enough extended vectors.

We explain the GIT construction of $X_{\mathbf{\Sigma}}$ from the extended
stacky fan $\mathbf{\Sigma}=(\bN, \Sigma,\beta,S)$ satisfying (C1-C3).
First we define a free $\ZZ$-module $\LL$ by the exact sequence
\begin{equation}\label{eq:exact}
  \xymatrix{
    0 \ar[r] &
    \LL \ar[r] &
    \ZZ^m \ar[r]^\beta & 
   \bN  \ar[r] & 
    0 }
\end{equation}
and define $K := \LL\otimes \Cstar$. 
The dual of \eqref{eq:exact} is an exact sequence:
\begin{equation} 
\label{eq:divseq}
\xymatrix{
  0 \ar[r] &
  \bN^{\vee} \ar[r] &
  (\ZZ^m)^{\vee} \ar[r]& 
  \LL^{\vee}}
\end{equation} 
and we define the character $D_i \in \LL^\vee$ of $K$ to be the image 
of the $i$th standard basis vector in $(\ZZ^m)^\vee$ under 
the third arrow $(\ZZ^m)^\vee \to \LL^\vee$.  
Set
\[
\cA_{\omega}=\left\{I\subset \{1,2,\cdots,m\} \mid
S \subset I, \ \text{$\sigma_{\overline{I}}$ 
is a cone of $\Sigma$}\right\} .
\]
to be the collection of \emph{anticones}.
The \emph{stability condition} $\omega\in \LL^{\vee}\otimes\RR$ 
lies in $\bigcap_{I \in \cA_{\omega}} \angle_{I}$, where
$$\angle_I = \big\{ \textstyle\sum_{i \in I} a_i D_i \mid \text{$a_i \in
      \RR$, $a_i > 0$} \big\} \subset
    \LL^\vee \otimes \RR.$$
The condition (C2) ensures that this intersection is non-empty.  
We understand $\angle_\emptyset=\{0\}$. Let 
\[
 U_{\omega}=\bigcup_{I\in \cA_\omega}(\CC^\times)^{I}\times \CC^{\overline{I}}
 := (\CC^\times)^{I}\times \CC^{\overline{I}}=\{(z_1,\cdots,z_m) \in \CC^m \mid
 z_i\neq 0~ \mbox{for}~ i\in I\}.
\]
The \emph{GIT data} consists of
\begin{itemize} 
\item $K \cong (\Cstar)^r$, a connected torus of rank $r$; 
\item $\LL = \Hom(\Cstar,K)$, the cocharacter lattice of $K$; 
\item $D_1,\ldots,D_m \in \LL^\vee = 
\Hom(K,\Cstar)$, characters of $K$;
\item stability condition $\omega\in  \LL^{\vee}\otimes\RR$;
\item $\cA_\omega= 
    \big\{ I \subset \{1,2,\ldots,m\} : \omega \in \angle_I \big\}$.
\end{itemize} 
The stability condition $\omega$ satisfies the following assumptions:
\begin{assumption} \label{assumption}
  \begin{enumerate}
  \item[(A1)] $\{1,2,\ldots,m\} \in \cA_\omega$;
  \item[(A2)] for each $I \in \cA_\omega$, 
the set $\{D_i : i \in I\}$ spans $\LL^\vee \otimes \RR$ over $\RR$.
\end{enumerate}
\end{assumption}
(A1) ensures that $X_\omega$ is non-empty; 
(A2) ensures that $X_\omega$ is a DM stack. 
Under these assumptions, $\cA_\omega$ is closed under 
enlargement of sets; i.e., if $I\in \cA_\omega$ and 
$I\subset J$ then $J \in \cA_\omega$. 
The {toric DM stack} 
is the quotient stack
$X_{\mathbf{\Sigma}}=X_\omega=[U_\omega/K]$.

Conversely, to obtain an extended stacky fan from GIT data, 
consider the exact sequence
(\ref{eq:exact}). 
Let $b_i = \beta(e_i)\in \bN$ and $\overline{b}_i\in \bN\otimes \RR$ 
be as above and, given 
a subset $I$ of $\{1,\dots,m\}$, let $\sigma_I$ denote the cone in 
$\bN\otimes \RR$ generated by $\{\overline{b}_i : i\in I\}$. 
The extended stacky fan
$\mathbf{\Sigma}_\omega=(\bN, \Sigma_\omega, \beta,S)$
corresponding to our data 
consists of the group $\bN$ and the map $\beta$ defined 
above, together with a fan $\Sigma_\omega$ in $\bN\otimes\RR$ 
and $S$ given by 
\begin{align*} 
\Sigma_\omega  = \{\sigma_{I} : \overline{I} \in \cA_\omega\}, 
 \qquad S  = \{ i \in \{1,\dots,m\} : \overline{\{i\}}  
\notin \cA_\omega\}. 
\end{align*} 
The quotient construction in \cite[\S2]{Jiang} 
coincides with the GIT quotient construction, 
and therefore $X_\omega$ is the toric DM 
stack corresponding to $\mathbf{\Sigma}_\omega$.

\subsection{Wall crossing and birational transformation}\label{sec:CRC:toric:DM:stacks}

The space $\LL^\vee \otimes \RR$ of stability conditions is divided into 
chambers by the closures of the sets $\angle_I$, $|I| = r-1$, 
and the DM stack $X_\omega$ depends on $\omega$ 
only via the chamber containing $\omega$.  
For any stability condition $\omega$, 
the set $U_\omega$ contains the big torus $T=(\Cstar)^m$.
Thus for any two such stability conditions $\omega_1$,~$\omega_2$ 
there is a canonical birational map 
$X_{\omega_1} \dashrightarrow X_{\omega_2}$, 
induced by the identity transformation between 
$T/K \subset X_{\omega_1}$ and 
$T/K \subset X_{\omega_2}$.  

Let $C_+$,~$C_-$ be chambers 
in $\LL^\vee \otimes \RR$ that are separated by a hyperplane wall $W$, 
so that $W \cap \overline{C_+}$ is a facet of $\overline{C_+}$,
$W \cap \overline{C_-}$ a facet of $\overline{C_-}$, 
and $W \cap \overline{C_+} = W \cap \overline{C_-}$.  
Choose stability conditions $\omega_+ \in C_+$, $\omega_- \in C_-$ 
satisfying (A1-A2) and set $U_+ := U_{\omega_+}$, $U_- := U_{\omega_-}$,
$X_+ := X_{\omega_+}$, $X_- := X_{\omega_-}$, and
\begin{align*}
  & \cA_\pm := \cA_{\omega_{\pm}} = 
\big\{ I \subset \{1,2,\ldots,m\} : 
\omega_\pm \in \angle_I \big\} .
\end{align*}
Then $C_\pm = \bigcap_{I\in \cA_\pm} \angle_I$. 
Let $\varphi \colon X_+ \dashrightarrow X_-$ be 
the birational transformation induced by the toric wall-crossing from $C_+$ to $C_-$
and suppose that $\sum_{i=1}^m D_i \in W$ which implies that $\varphi$ is crepant. 
Let $e \in \LL$ denote the \emph{primitive lattice vector} in $W^\perp$ 
such that $e$ is positive on $C_+$ and negative on $C_-$. 
We fix the notations
\begin{itemize} 
\item  $M_{+}:=\{i\in\{1,\cdots, m\}| D_i\cdot e>0\}$,
\item  $M_{-}:=\{i\in\{1,\cdots, m\}| D_i\cdot e<0\}$,
\item  $M_{0}:=\{i\in\{1,\cdots, m\}| D_i\cdot e=0\}$.
\end{itemize}

Choose $\omega_0$ from the relative interior of $W\cap \overline{C_+} 
= W \cap \overline{C_-}$. The stability condition $\omega_0$ 
does not satisfy (A1-A2) on GIT data, but consider 
\begin{align*} 
\cA_0 & := \cA_{\omega_0} = 
\left\{ I \subset \{1,\dots,m\} : 
\omega_0 \in \angle_I \right\} 
\end{align*} 
and the corresponding toric Artin stack $X_0 := X_{\omega_0} =[U_{\omega_0}/K]$.
Here $X_0$ is not a DM stack, as 
the $\Cstar$-subgroup of $K$ 
corresponding to $e\in \LL$ (the defining equation 
of the wall $W$) has a fixed point in $U_0 := U_{\omega_0}$. 
The stack $X_0$ contains both $X_+$ and $X_-$ as open substacks 
and the canonical line bundles of $X_{+}$ and $X_-$ 
are the restrictions of the same line bundle $L_0\to X_0$ 
given by the character ${-\sum_{i=1}^m} D_i$ of $K$. 
The condition $\sum_{i=1}^m D_i\in W$ ensures that $L_0$ comes from 
a $\QQ$-Cartier divisor on the underlying singular toric variety 
$\overline{X}_0 = \CC^m/\!\!/_{\omega_0} K$.  
There are canonical blow-down maps $g_\pm \colon X_\pm \to \overline{X}_0$, 
and $K_{X_\pm}=g_\pm^\star L_0$.  We have a commutative diagram:
\begin{equation}\label{eq:K-equivalence}
\xymatrix{
  &~\tX  \ar[rd]^{f_-} \ar[ld]_{f_+} & \\ 
  X_+ \ar[rd]_{g_+} \ar@{-->}^{\varphi}[rr] &  & X_- \ar[ld]^{g_-} \\ 
  & \overline{X}_0 & 
}
\end{equation}
This shows that $f_+^\star(K_{X_+}) = f_-^\star(K_{X_-})$ and
birational map $\varphi$ is \emph{crepant}, since they 
are the pull-backs of the same $\QQ$-Cartier divisor on $\overline{X}_0$. 

To construct $\tX$, consider the action of $K \times \Cstar$ on $\CC^{m+1}$ defined by the characters 
$\tD_1,\ldots,\tD_{m+1}$ of $K \times \Cstar$, where:
\[
\tD_j = 
\begin{cases}
  D_j \oplus 0 & \text{if $j < m+1$ and $D_j \cdot e \leq 0$} \\
  D_j \oplus ({-D_j} \cdot e) & \text{if $j < m+1$ and $D_j \cdot e > 0$} \\
  0 \oplus 1 & \text{if $j=m+1$}
\end{cases}
\]
Consider the chambers $\widetilde{C}_+$,~$\widetilde{C}_-$, 
and~$\widetilde{C}$ in $(\LL \oplus \ZZ)^\vee \otimes \RR$ 
that contain, respectively, the stability conditions
\begin{align*}
  \tomega_+ = (\omega_+,1) &&
  \tomega_- = (\omega_-,1) && \text{and} && 
  \tomega = (\omega_0, - \varepsilon)
\end{align*}
where $\varepsilon$ is a very small positive real number.  
Let $\tX$ denote the toric DM stack 
defined by the stability condition $\tomega$.  
We have, by \cite[Lemma~6.16]{CIJ}, that
the toric DM stack corresponding  to the chamber $\widetilde{C}_{\pm}$ is $X_{\pm}$. 
Furthermore, there is a commutative diagram as in \eqref{eq:K-equivalence}, 
where:
   $f_{\pm} \colon \tX \to X_{\pm}$ is a toric blow-up, 
   arising from the wall-crossing from $\widetilde{C}$ to $\widetilde{C}_{\pm}$.

\section{Generalized Bondal-Orlov transforms}\label{BO}

\subsection{Equivariant $K$-theory of toric DM stacks}

The big torus $T:=(\CC^\times)^m$ acts on the toric DM stack 
$X_\omega$ corresponding to a stability condition 
$\omega\in \LL^\vee\otimes\RR$ satisfying assumptions (A1-A2). 
The $T$-equivariant $K$-theory group 
$K_0^{T}(X_\omega)$ of $X_\omega$ is generated by the $T$-equivariant line 
bundles $R_i$ corresponding to the ray $\rho_i$ for each $i\in\{1,\cdots,m\}$. 

Recall that the \emph{torus fixed points of $X_\omega$ are in one-to-one 
correspondence with minimal anticones $\delta\in\cA_\omega$}.
 A minimal anticone $\delta$ determines a torus fixed point stack 
$x_{\delta}=BG_\delta\in X_\omega$, where $G_\delta$ is the isotropy group of 
the fixed point $x_{\delta}$.  
Let $i_\delta \colon x_\delta \to X_\omega$ denote the inclusion. We have
\begin{equation} \label{e:2.1}
 i_{\delta}^* R_j =1, \quad \forall j \in \delta.
\end{equation}

We recall the Lefschetz fixed point theorem
(c.f.\ \cite[Theorem 3.3]{CIJS} in this formulation).

\begin{theorem}
\label{theorem:localization}
Let $X_\omega = [U_\omega/K]$ be a toric DM stack.
The torus $T$ acts on $X_\omega$.  
Given $\delta \in \cA_\omega$, write $x_\delta$ for 
the corresponding $T$-fixed point of $X_\omega$.
Let $N_\delta$ denote the normal bundle to $i_\delta$. 
Let $\ZZ[T] = K_T^0({pt})$ denote the ring of 
regular functions (over $\ZZ$) on $T$ and let $\on{Frac}\ZZ[T]$ 
denote the field of fractions. 
Then for $\alpha \in K_T^0(X_\omega)$, we have 
\begin{align*}
\alpha = \sum_{\delta\in \cA_\omega} (i_\delta)_\star 
\left( 
\frac{i_\delta^\star \alpha}{\lambda_{-1} N_\delta^\vee}
\right) 
 \quad \in K_T^0(X_\omega)\otimes_{\ZZ[T]} \on{Frac}(\ZZ[T])
\end{align*}
where $\lambda_{-1} N_\delta^\vee  
:= \sum_{i=0}^{\dim X_\omega} (-1)^i \bigwedge^i N_\delta^\vee$ 
is invertible in $K_T^0(x_\delta)\otimes_{\ZZ[T]} 
\on{Frac}(\ZZ[T])$. 
\end{theorem}

\subsection{The localized $K$-theory basis}

Consider the toric wall crossing diagram \eqref{eq:K-equivalence}. The torus $T$ acts on $X_\pm$ through the diagonal action of $T$ on 
$\CC^m$. There is an action of $T$ on $\tX$ induced from the inclusion $T=T\times \{1\}\subset T\times \CC^\times$ and the $T\times\CC^\times$ action on $\CC^{m+1}$. So all the maps in \eqref{eq:K-equivalence} are $T$-equivariant.
The $T$-equivariant K-groups $K_0^{T}(X_\pm)$, $K_0^{T}(\tX)$ are modules over $K_0^{T}(pt)=\ZZ[T]$.

From the wall crossing construction in \S \ref{sec:CRC:toric:DM:stacks}, 
there are  two types of minimal anticones for $\widetilde{X}$.  
The first type, called \emph{flopping type}, is given by 
$\widetilde{\delta}=(j_{1},\cdots, j_{r-1}, j_{+}, j_{-}),$
where $j_1,\cdots,j_{r-1}\in M_0$, and $j_{+}\in M_+, j_{-}\in M_{-}$. 
This type of minimal anticones induce the maps from the fixed point stack 
of $\tX$ to the fixed point stacks of $X_{+}$ and $X_{-}$ by
\[
 f_{+,\widetilde{\delta}}: x_{\widetilde{\delta}}\to x_{\delta_{+}}, \quad 
 f_{-, \widetilde{\delta}}: x_{\widetilde{\delta}}\to x_{\delta_{-}},
\]
where $\delta_{+}=(j_{1},\cdots, j_{r-1}, j_{+}, m+1)$ and 
$\delta_{-}=(j_{1},\cdots, j_{r-1}, j_{-}, m+1)$.
We use the following notations:
$\widetilde{\delta}|\delta_{\pm}$ means that the fixed point 
$x_{\widetilde{\delta}}$ maps to the fixed point $x_{\delta_{\pm}}$ corresponding to
flopping minimal anticone $\delta_{\pm}$ for $X_{\pm}$.  

The second type of minimal anticone, called \emph{nonflopping type}, 
is given by $\widetilde{\delta}$ containing the last, $m+1$-st, ray 
corresponding to the common blow-up. 
The nonflopping minimal anticones map isomorphically to minimal 
anticones of $X_{+}$ and $X_{-}$.
Such minimal anticones ($\widetilde{\delta}$ and $\delta_{\pm}$) are of the form
$(j_{1},\cdots, j_{r-2}, j_+, j_{-}, m+1)$.

The  $T$-invariant divisor $\{z_i=0\}$ on $X_\omega$  determines a $T$-equivariant line bundle
$\sO(\{z_i=0\})$ on $X_\omega$, and we denote the class of this line bundle in the 
$T$-equivariant $K$-theory by $R_i$. For  $K_0^{T}(X_\pm)$, $K_0^{T}(\tX)$ we write these  classes as:
\[
\begin{array}{ll}
 \{R^{-}_i| 1\leq i\leq m\}: & \mbox{for} ~K^{T}_0(X_{-});\\
 \{R^{+}_i| 1\leq i\leq m\}: & \mbox{for} ~K^{T}_0(X_{+});\\ 
\{\widetilde{R}_i| 1\leq i\leq m+1\}: & \mbox{for} ~K^{T}_0(\tX).\\ 
\end{array}
\]

From \S 6.3.2 in \cite{CIJ}, each character $p\in \mathbb{H}om(K,\CC^\times)=\LL^\vee$ define a   line bundle 
$L_-(p)$ over $X_-$.    This line bundle $L_{-}(p)$ is equipped with a $T$-linearized action, thus make it a 
$T$-equivariant line bundle. 
The line bundles 
$R_i^-=L_-(D_i)\otimes e^{\lambda_i}$, where  $e^{\lambda_i}$ is the standard $i$-th 
irreducible $T$-representation $T\to\CC^\times$. 
Similar construction works for the $K$-theory ring $K_0^T(X_+)$.

For a character $(p,n)\in \Hom(K\times \CC^\times,\CC^\times)=\LL^\vee\oplus \ZZ$ we define a $T$-equivariant line bundle 
$L(p,n)\to\tX$ ans we have:
$$\widetilde{R}_i=L(\tD_i)\otimes e^{\lambda_i}, (1\leq i\leq m); \quad \widetilde{R}_{m+1}=L(\tD_{m+1})=L(0,1).$$
The classes $L_\pm(X_\pm)$ (the classes $L(p,n))$) generate the equivariant $K$-group
$K^{T}_0(X_{\pm})$ ($K^{T}_0(\tX)$) over $\ZZ[T]$.

We describe the localized $T$-equivariant $K$-theory basis for $K_0^T(X_-)$.  Let 
$\delta_-\in\cA_-$ be a minimal cone and 
$x_{\delta_-}$ be the corresponding $T$-fixed point.  Let
$$i_{\delta_-}: x_{\delta_-}\to X_-$$
be the inclusion of the fixed point, and $G_{\delta_-}$ the isotropy group of $x_{\delta_-}$. We have
$x_{\delta_-}=BG_{\delta_-}$.  A basis for $K_0^T(X_-)$, after inverting nonzero elements of $\ZZ[T]$, is given by
\begin{equation}\label{K-basis-}
\{(i_{\delta_-})_{\star}\varrho: \varrho \mbox{~an irreducible representation of ~}G_{\delta_-}, \delta_-\in\cA_- \}
\end{equation}
Choose a lift $\hat{\varrho}\in\Hom(K,\CC^\times)=\LL^\vee$ of each $G_{\delta_-}$-representation $\varrho: G_{\delta_-}\to\CC^\times$, an element in 
(\ref{K-basis-}) can be written in the form:
$$e_{\delta_-,\varrho}:=L_-(\hat{\varrho})\prod_{i\notin\delta_-}(1-S_i^{-}).$$
Then $\{e_{\delta_-,\varrho}\}$ is a basis for the localized $T$-equivariant $K$-theory of $X_-$. There is a similar basis
$\{e_{\delta_+,\varrho}\}$
for the localized $T$-equivariant $K$-theory of $X_+$.

\subsection{The Bondal-Orlov functors}

The general Bondal-Orlov  functor
on the bounded derived categories $D^b_{T}(X_\pm)$:

$$\mathbb{BO}_k:  D^{b}_{T}(X_{-})\to D^b_{T}(X_{+})$$ 
is defined by:
$$\mathbb{BO}_k(\alpha)=(f_{+})_{\star}(\sO_{\tX}(kE)\otimes (f_{-})^{\star}(\alpha)).$$

We consider the induced functor on the $K$-theory of $X_\pm$: 
$$\mathbb{BO}_k:  K_0^{T}(X_{-})\to K_0^{T}(X_{+}).$$ 
We explicitly calculate $\mathbb{BO}_k$ in terms of the  localized $T$-equivariant $K$-theory basis for $X_-$. 
Let
$$S_i^{+}:=(R_{i}^{+})^{-1}, \quad  S_{i}^{-}:=(R_{i}^{-})^{-1},  \quad \widetilde{S}_{i}:=(\widetilde{R}_{i})^{-1},$$  
and let 
$$k_i:=\mbox{max}(D_i\cdot e, 0), \quad  l_i:=\mbox{max}(-D_i\cdot e, 0).$$

\begin{theorem}\label{theorem:GFM}
Let $\delta_-\in\cA_-$ be a minimal anticone such that $\delta_-\in\cA_+$, then 
$\mathbb{BO}_k(e_{\delta_-,\varrho})=e_{\delta_-,\varrho}$, where on the right side $\delta_-$ is taken as a minimal anticone in $\cA_+$;
If $\delta_-\in\cA_-$ is a minimal anticone such that $\delta_-\notin\cA_+$, then 
\begin{align*}
&\mathbb{BO}_k(e_{\delta_-,\varrho})=\\
&\frac{1}{l}\sum_{t\in\mathcal{T}}\left(\frac{t^k(1-S_{j_-}^{+})}{1-t^{-1}}\cdot L_+(\hat{\varrho})t^{\hat{\varrho}\cdot e}\cdot
\prod_{\substack{j\notin\delta_- \\
D_j\cdot e<0}}(1-S_j^{+})\cdot \prod_{\substack{i\notin\delta_- \\
D_i\cdot e\geq 0}}(1-t^{-D_i\cdot e}S_i^{+}) \right)
\end{align*}
where $j_-\in\delta_-$ is the unique element such that $D_{j_-}\cdot e<0$,
$l=-D_{j_-}\cdot e$ and 
$$\mathcal{T}:=\{\zeta\cdot (R_{j_-}^{+})^{\frac{1}{l}}: \zeta\in\mu_l\}.$$
\end{theorem}
\begin{proof}
The proof is similar to the proof of Theorem 6.19 in \cite{CIJ}, except that we take into account of the role of the line bundle 
$\sO_{\tX}(kE)$.  The line bundle $\sO_{\tX}(E)$ corresponds to the line bundle $\widetilde{R}_{m+1}$ over $\tX$. So 
$$\sO_{\tX}(kE)\cong \widetilde{R}_{m+1}^{\otimes k}.$$

We calculate $\mathbb{BO}_k$ for any $k\in\ZZ$.  
For $\delta_-\in\cA_\pm$, $\varphi$ is an isomorphism in an neighbourhoods of the fixed points of $x_{\delta_-}\in X_\pm$. So 
$\mathbb{BO}_k(e_{\delta_-,\varrho})=e_{\delta_-,\varrho}$.

Suppose now that $\delta_-\in\cA_-$, but $\delta_-\notin\cA_+$. Let 
$\delta_-=\{j_1,\cdots, j_{r-1}, j_-\}$. Then 
$D_{j_1}\cdot e=D_{j_2}\cdot e=\cdots=D_{j_{r-1}}\cdot e=0$ and $D_{j_-}\cdot e<0$. 
We have from \cite[Proposition 6.21]{CIJ}, 
$$(f_-)^\star(e_{\delta_-,\varrho})=L(\hat{\varrho},0)\prod_{i\notin\delta_-}(1-\widetilde{S}_{m+1}^{k_i}\widetilde{S}_i).$$
Then 
$$\sO_{\tX}(kE)\otimes (f_-)^\star(e_{\delta_-,\varrho})=\widetilde{R}_{m+1}^{k}\cdot L(\hat{\varrho},0)\prod_{i\notin\delta_-}(1-\widetilde{S}_{m+1}^{k_i}\widetilde{S}_i).$$
We use the localized Theorem \ref{theorem:localization} in the $T$-equivariant $K$-theory restricting above to all torus fixed points 
$x_{\widetilde{\delta}}\in f_-^{-1}(x_{\delta_-})$, where $\widetilde{\delta}=\delta_-\cup\{j_+\}$ for $D_{j_+}\cdot e>0$.
So
\begin{equation}\label{localization-deltatilde}
\sO_{\tX}(kE)\otimes (f_-)^\star(e_{\delta_-,\varrho})=\sum_{\widetilde{\delta}\in\widetilde{\cA}}
(i_{\widetilde{\delta}})_{\star}(i_{\widetilde{\delta}})^\star\Big[ \frac{\widetilde{R}_{m+1}^{k}\cdot L(\hat{\varrho},0)\cdot 
\prod_{i\notin \delta_-}(1-\widetilde{R}_{m+1}^{k_i}\widetilde{S}_i)}{(1-\widetilde{S}_{m+1})\prod_{\substack{j\notin\delta_-\\
j\neq j_+}}(1-\widetilde{S}_j)} \Big]
\end{equation}
For $j_+\in\widetilde{\delta}$, $\widetilde{R}_{j_+}$ is trivial when restricted to $x_{\widetilde{\delta}}$. So:
$(1-\widetilde{\delta}_{m+1}^{k_i}\widetilde{S}_{j_+})=(1-\widetilde{\delta}_{m+1}^{k_i})$ and (\ref{localization-deltatilde}) is actually a polynomial on $\widetilde{R}_{m+1}$ on the numerator. Then applying the pushforward
\begin{align*}
&(f_+)_{\star}(\sO_{\tX}(kE)\otimes (f_-)^\star(e_{\delta_-,\varrho}))\\
&=\sum_{\delta_+: \delta_+|\delta_-}
(i_{\delta_+})_{\star}(f_{+,\widetilde{\delta}})_\star (i_{\widetilde{\delta}})^\star\Big[ \frac{\widetilde{R}_{m+1}^{k}\cdot L(\hat{\varrho},0)\cdot 
\prod_{i\notin \delta_-}(1-\widetilde{R}_{m+1}^{k_i}\widetilde{S}_i)}{(1-\widetilde{S}_{m+1})\prod_{\substack{j\notin\delta_-\\
j\neq j_+}}(1-\widetilde{S}_j)} \Big] \\
&=\sum_{\delta_+: \delta_+|\delta_-}
(i_{\delta_+})_{\star} (i_{\delta_+})^\star\Big[\frac{1}{l}\sum_{t\in\mathcal{T}}\frac{t^{k}\cdot L_+(\hat{\varrho})\cdot t^{\hat{\varrho}\cdot e}
\prod_{i\notin \delta_-}(1-t^{l_i-k_i}S_i^+)}{(1-t^{-1})\prod_{\substack{j\notin\delta_-\\
j\neq j_+}}(1-t^{l_j}S^+_j)} \Big] 
\end{align*}
here we use the formula (3) in Proposition 6.22 of \cite{CIJ}. Hence we get:
\begin{align*}
&(f_+)_{\star}(\sO_{\tX}(kE)\otimes (f_-)^\star(e_{\delta_-,\varrho})) \\
&=
\sum_{\delta_+: \delta_+|\delta_-}
(i_{\delta_+})_{\star} (i_{\delta_+})^\star\Big[\frac{\frac{1}{l}\sum_{t\in\mathcal{T}}\frac{t^{k}\cdot (1-S_{j_-}^+)}{1-t^{-1}}\cdot L_+(\hat{\varrho})\cdot t^{\hat{\varrho}\cdot e}
\prod_{i\notin \delta_-}(1-t^{-k_i}S_i^+)}{\prod_{\substack{j\notin\delta_+\\
j\neq j_+}}(1-S^+_j)} \Big] 
\end{align*}
By localization again we get the result in the Theorem. The only thing we need to check is that 
$t^{k}\cdot (1-S_{j_-}^+)\cdot L_+(\hat{\varrho})\cdot t^{\hat{\varrho}\cdot e}
\prod_{i\notin \delta_-}(1-t^{-k_i}S_i^+)$ vanishes on $x_\delta$ for $\delta\in\cA_{+}\cap\cA_-$. 
But this is a similar check as in the proof of Theorem 6.19 of \cite{CIJ}.
\end{proof}

\begin{remark}
In Theorem 6.23 of \cite{CIJ}, we prove that $\mathbb{BO}_0$ actually matches the analytic continuation of $I$-functions of $X_\pm$.  Since the $I$-functions of $X_\pm$ determine the bid quantum cohomology for $X_\pm$, the result of Theorem 6.23 in \cite{CIJ} tells us that the Fourier--Mukai transform preserves the big quantum cohomology of a single toric wall crossing. 
It is of course interesting to see if the general Bondal-Orlov transforms  $\mathbb{BO}_k$ preserves some analytic continuation of hypergeometric function of $X_\pm$. 
\end{remark}

\section{Derived equivalence and spherical twists}\label{derived:equivalence}

In this section we recall some facts that the general Bondal-Orlov functors give  equivalences on the bounded  derived categories. 

\subsection{Derived equivalence}
Let $Q:=T/K$ be the quotient torus since $K\subset T$ is a subtorus.  Both $X_+$ and $X_-$ carry  effective actions of $Q$. 
In this section we prove the following:
\begin{theorem}
Let (\ref{eq:K-equivalence}) be a toric crepant transformation. Then 
$$\mathbb{BO}_k: D_{Q}^b(X_-)\to D_{Q}^b(X_+)$$
gives an equivalence on the equivariant  bounded derived categories. 
\end{theorem}
\begin{remark}
We use the same proof as in \cite[\S 5]{CIJS}, which uses the idea of Halpern-Leistner \cite{HL} and Halpern-Leistner-Shipman \cite{HLS}.
\end{remark}

\begin{proof}
We mainly follow the construction and notations in \S 5 of \cite{CIJS}.  First we recall the variation of the GIT quotients of $X\pm$ and $\tX$.  They correspond to chambers $\widetilde{C}_\pm, \widetilde{C}$ inside $(\LL^\vee\oplus\ZZ)\otimes \RR$. 
We denote by the walls by 
$W_{+|-}, W_{+|\sim}, W_{-|\sim}$ respectively.  Let
$$W_0=W_{+|-}\cap W_{+|\sim}\cap W_{-|\sim}.$$
There are $7$ stability conditions on $W_0, \widetilde{C}_\pm, \widetilde{C}, W_{+|-}, W_{+|\sim}, W_{-|\sim}$ respectively. 
If we let $V_0\subset \CC^{m+1}$ be the semi-stable locus of $W_0$, then 
$$V_0=U_0\times \CC=\CC^{m+1}\setminus \left(\cup_{I\notin\cA_0}\CC^{I}\times\CC\right)$$
where $U_0$ is in \S \ref{sec:CRC:toric:DM:stacks}. As in \cite{CIJS}, the other $6$ stability conditions are as follows:
\begin{center}
  \begin{tabular}{ll}
    \multicolumn{1}{c}{Location of  stability condition} & \multicolumn{1}{c}{Semi-stable locus}   \\
    \hspace{2.25cm}$\widetilde{C}_+$ 	& \hspace{2ex}	$V_+ = V_0\setminus \left((\CC^{M_{\leq 0}}\times \CC)\cup \CC^m\right)	$		\\
    \hspace{2.25cm}$\widetilde{C}_-$ 	& \hspace{2ex}	$V_- = V_0\setminus \left((\CC^{M_{\geq 0}}\times \CC)\cup \CC^m\right)	$				\\
    \hspace{2.25cm}$\widetilde{C}$ 	& \hspace{2ex}	$V_{\sim} = V_0\setminus \left((\CC^{M_{\leq 0}}\times \CC)\cup (\CC^{M_{\geq 0}}\times \CC)\right)$	\\
    \hspace{2.25cm}$W_{+|-}$		& \hspace{2ex}	$V_{+|-} = V_0\setminus \CC^m$		\\
    \hspace{2.25cm}$W_{+|\sim}$ 	& \hspace{2ex}	$V_{+|\sim}= V_0\setminus (\CC^{M_{\leq 0}}\times \CC)$		\\
    \hspace{2.25cm}$W_{-|\sim}$ 	& \hspace{2ex}	$V_{-|\sim} = V_0\setminus (\CC^{M_{\geq 0}}\times \CC)$			
  \end{tabular}
\end{center}
We have the  GIT quotients 
$$X_+=\big[V_+ /K\big],$$
$$X_-=\big[V_- /K\big],$$
$$\tX=\big[V_{\sim} /K\big].$$

Now we recall the $KN$ stratum introduced in \cite[\S 5.1]{CIJS}. A $KN$ stratum $(\lambda, Z, S)$ contains a one-parameter subgroup 
$\lambda\subset K\times\CC^\times$, a connected component $Z$ of the fixed locus, and the associated blade $S$ defined as:
$$S=\{y\in\CC^{m+1}: \lim_{t\to \infty}\lambda(t)(y)\in Z\}.$$ 
To a $KN$ stratum, there is a numerical invariant
$$\eta=\mbox{Weight}_{\lambda}(\mbox{det}(N_{S/\CC^{m+1}})).$$
In our cases let 
$$d:=\sum_{i\in M_+}D_i\cdot e=-\sum_{i\in M_-}D_i\cdot e$$
and consider the $KN$-strata:
$$((e,1), \CC^{M_{\geq 0}}\cap V_{+|\sim}, \CC^m\cap V_{+|\sim}),  \quad \eta=1$$
and 
$$((-e,-1), \CC^{M_{\geq 0}}\cap V_{+|\sim}, \CC^m\cap V_{+|\sim}), \quad  \eta=d$$
Then $V_+$ and $V_{\sim}$ are open subsets of $V_{+|\sim}$, which are the complements of the above $KN$ strata. 
Then from \cite[\S 5]{CIJS} and \cite{HL}, let 
$$\mathbf{F}\subset \widetilde{\mathbf{F}}\subset D^b_{T\times\CC^\times}(V_{+|\sim})$$
be the subcategories by imposing the grade-restriction rule 
on the subvariety $\CC^{M_{\geq 0}}\cap V_{+|\sim}$, where for  $\mathbf{F}$ we require that the $(e,1)$-weights lie in $[0,1)$, and for 
$\widetilde{\mathbf{F}}$ we we require that the $(e,1)$-weights lie in $[0,d)$.  The we have the following diagram: 
\[
\xymatrix{
&\widetilde{\mathbf{F}}\ar[dl]_{\cong}\ar[dr]\\
D^b_{Q}(\tX)\ar[rr]^{(f_+)_{\star}}&&D^b_{Q}(X_+)
}
\]
and the diagonal map is the restriction of functors.  Similarly, take $V_-$ as an open subset of $V_{-|\sim}$ and taking into account of the $KN$-stratum:
$$((0,1), \CC^{M_{\leq 0}}\cap V_{-|\sim}, \CC^m\cap V_{-|\sim})$$
which has numerical invariant $\eta=1$. There is a subcategory
$$\mathbf{H}\subset D^b_{T\times\CC^\times}(V_{-|\sim})$$
such that the $(0,1)$-weights lie in $[0,1)$. We have the commuting triangle:
\[
\xymatrix{
&\mathbf{H}\ar[dl]_{\cong}\ar[dr]\\
D^b_{Q}(X_-)\ar[rr]^{(f_-)^{\star}}&&D^b_{Q}(\tX)
}
\]
and the diagonal map is the restriction of functors

Let us recall the definition of the functor 
$\mathbb{GR}_k$ for each integer $k\in\ZZ$ in \cite[\S 5.1]{CIJS}. Note that \cite{CIJS} only discusses the case $\mathbb{GR}_0$, but general $\mathbb{GR}_k$ are similar.  For the $KN$ stratum $(e,Z,S_-)$ with numerical invariant $\eta_+=\sum_{i\in M_+}D_i\cdot e$, 
$Z=U_0\cap \CC^{M_0}$ and $S_-=U_0\cap \CC^{M_{\leq 0}}$, where 
the toric DM stack $X_+=[(U_0\setminus S_-)/K]$, 
there exists a subcategory
$$\mathbf{G}_k\subset D^b_{T}(U_0)$$
using the grade restriction rule and requiring the $e$-weights lying in $[k,k+\eta_+)$.  We have
$\mathbf{G}_k\cong D^b(X_+)$. 

On the other hand, for the $KN$ stratum $(-e,Z,S_+)$ with numerical invariant $\eta_-=-\sum_{i\in M_-}D_i\cdot e$, 
$Z=U_0\cap \CC^{M_0}$ and $S_+=U_0\cap \CC^{M_{\geq 0}}$, where the toric DM stack $X_-=[(U_0\setminus S_+)/K]$,  there exists a subcategory
$$\mathbf{G}_k\subset D^b_{T}(U_0)$$
using the grade restriction rule and requiring the $(-e)$-weights lying in $[-\eta_-+k+1, k+1)$.  Then we have
$\mathbf{G}_k\cong D^b(X_-)$. 
Thus the functor $\mathbb{GR}_k: D^b_{Q}(X_-)\to D^b_{Q}(X_+)$
 are defined by the diagram:
\[
\xymatrix{
&\mathbf{G}_k\ar[dl]_{\cong}\ar[dr]^{\cong}\\
D^b_{Q}(X_-)\ar[rr]^{}&&D^b_{Q}(X_+)
}
\] 
by inverting the right isomorphism.  

Consider the subcategory
$(\pi_-)^\star \mathbf{G}_0\subset D^b_{T\times\CC^\times}(V_0)$, where 
$$\pi_-: \Big[V_0\Big/(T\times\CC^\times)\Big]\to [U_0/T]$$ is the natural morphism. 
Under the restriction functor from 
$V_0\to V_{+|\sim}$, the subcategory $(\pi_-)^\star \mathbf{G}_0$ maps  to $\widetilde{\mathbf{F}}$. 
Under the restriction functor from 
$V_0\to V_{-|\sim}$, the subcategory $(\pi_-)^\star \mathbf{G}_0$ maps  to $\mathbf{H}$, which is an isomorphism. 

The line bundle $\sO_{\tX}(kE)\to \tX$ corresponds to an $T\times\CC^\times$-equivariant line bundle $\mathcal{L}_k$ on $\CC^{m+1}$. Let 
$$\otimes\mathcal{L}_k: D^b_{Q}(\tX)\to D^b_{Q}(\tX)$$
be the tensor product morphism. Then since the line bundle has $e$-weight $k$, the tensor product sends 
$(\pi_-)^\star \mathbf{G}_0$ to $(\pi_-)^\star \mathbf{G}_k$.  We have the following modified diagram as for the last diagram in \cite{CIJS}:
\[ 
\xymatrix{ & & (\pi_-)^\star\mathbf{G}_0\ar[dl]_{\simeq}\ar[r]^{\otimes \mathcal{L}_k} &(\pi_-)^\star \mathbf{G}_k \ar[dr] & & \\
& \mathbf{H} \ar[dl]_{\simeq} \ar[dr] & && \widetilde{\mathbf{F}} \ar[dl]_{\simeq} \ar[dr] & \\
D^b_Q(X_-) \ar[rr]^{(f_-)^*} & & D^b_Q(\tilde{X})\ar[r]^{\otimes \mathcal{L}_k}&D^b_Q(\tilde{X}) \ar[rr]^{(f_+)_\star} & & D^b_Q(X_+)}
\]
The result is easily seen from the above diagram since the bottom represents the Bondal-Orlov functor $\mathbb{BO}_k$. 
\end{proof}

\subsection{The spherical twist}

Let us fix a single toric wall crossing (\ref{eq:K-equivalence}).  We first classify the exceptional locus of the contractions 
$g_\pm$.  

For $g_+: X_+\to \overline{X}_0$, let 
$$\LL^\vee_{\ex}:=\LL^\vee/\langle D_i: i\in M_0\rangle$$
and let $p: \LL^\vee\to \LL^\vee_{\ex}$ be the projection. Then  $p: \LL^\vee\otimes\RR\to \LL^\vee_{\ex}\otimes\RR$ is the projection to the vector spaces.  Let $\omega^{+}_{\ex}=p(\omega_+)$ be the image of the stability condition $\omega_+$.  
The lattice $\LL^\vee_{\ex}$, which is rank one,  may have torsion in general.  In this section we assume that $\langle D_i: i\in M_0\rangle$ generate the lattice wall $W\cap\LL^\vee$. Then $\LL^\vee_{\ex}\cong \ZZ$.  
The elements $D_i\in \LL^\vee$ have images 
$p(D_i)=D_i\cdot e\in\LL^\vee_{\ex}$. So only $D_i$ for $i\in M_{\pm}$ survive  Hence we get the GIT data on $\LL^\vee_{\ex}$:
\begin{itemize} 
\item $K \cong \Cstar$, a connected torus of rank $1$; 
\item $\LL_{\ex}= \Hom(\Cstar,K)$; 
\item $D_1\cdot e,\ldots,D_m\cdot e \in \LL_{\ex}^\vee = 
\Hom(K,\Cstar)$, characters of $K$;
\item stability condition $\omega_{\ex}^+\in  \LL_{\ex}^{\vee}\otimes\RR$;
\item $\cA_{\omega^+_{\ex}}= 
    \big\{ I \subset \{1,2,\ldots,m\} : D_i\cdot e>0, i\in I \big\}$.
\end{itemize} 
Let $a_i:=D_i\cdot e$ for $D_i\cdot e>0$ and $\mathbf{a}=(D_i\cdot e: D_i\cdot e>0)$.
\begin{proposition}
The corresponding toric DM stack $X_{\omega_{\ex}^+}$ associated with the above GIT data is the weighted projective stack $\mathbb{P}(\mathbf{a})$.  Moreover, the map  $g_+: X_+\to \overline{X}_0$ always contracts the weighted projective stack $X_{\omega_{\ex}^+}=\mathbb{P}(\mathbf{a})$.
\end{proposition}
\begin{proof}
The first statement is easily seen from the GIT data. For the second statement, look at the map
$$g_+: X_+=[U_+/K]\to \overline{X}_0=\overline{[U_0/K]},$$
where $U_+=U_0\setminus (\CC^{M_{\leq 0}}\cap U_0)$.  
The torus $\CC^\times$-fixed points on $\overline{X}_0=\overline{[U_0/K]}$ corresponds to nonsimplicial cones, which are spanned by 
rays containing $D_i$'s for $i\in M_{\pm}$.  Then from the above map $g_+$, it must contract the weighted projective stack 
$\mathbb{P}(\mathbf{a})$ to this fixed point. 
\end{proof}
\begin{remark}
Similar result holds for the contract map 
$$g_-: X_-=[U_-/K]\to \overline{X}_0=\overline{[U_0/K]}.$$
Let $b_i:=D_i\cdot e$ for $D_i\cdot e<0$ and $\mathbf{b}=(D_i\cdot e: D_i\cdot e<0)$.
Then $g_-$ contracts the weighted projective stacks $\mathbb{P}(\mathbf{b})$.
\end{remark}

Let $N:=\sum_{i: D_i\cdot e>0}D_i\cdot e=-\sum_{i: D_i\cdot e<0}D_i\cdot e$, which is the sum of the weights.

\begin{proposition}\label{GR:BO}
We have for any $k\in\ZZ$,
$$\mathbb{GR}_k\cong \mathbb{BO}_{(N-1)+k}.$$
\end{proposition}
\begin{proof}
We generalize the proof of Proposition 3.1  in \cite{ADM}. We show that 
$\mathbb{GR}_k^{-1}\circ \mathbb{BO}_{(N-1)+k}$ takes $\sO_{X_-}(l)$ to  $\sO_{X_-}(l)$  and acts as identity on 
$$\mathcal{E}xt^i( \sO_{X_-}(l) , \sO_{X_-}(l^\prime))$$
for $k\leq l$, $l^\prime\leq k+(N-1)$, since thess objects split-generate the derived  category $D^b(X_-)$. 
First  $\mathbb{GR}_k$ takes $\sO_{X_-}(l)$ to  $\sO_{X_-}(-l)$
for $k\leq l\leq k+(N-1)$,  since the subcategory  $\mathbf{G}_k\subset D^b_{T}(U_0)=D^b(X_0)$ is the full-subcategory split-generated by 
$$\sO_{X_0}(k), \cdots, \sO_{X_0}(k+(N-1)).$$
Also
\begin{align*}
\mathbb{BO}_{(N-1)+k}(\sO_{X_-}(l))&=(f_+)_{\star}\left(\sO_{\tX}(((N-1)+k)E)\otimes (f_-)^\star(\sO_{X_-}(l)\right)\\
&=(f_+)_{\star}\left(\sO_{\tX}(l-(N-1)-k, -(N-1)-k)\right)\\
&=\sO_{X_+}(-l)\otimes (f_+)_{\star}\left(\sO_{\tX}(((N-1)+k-l)E)\right)
\end{align*}
and
$$(f_+)_{\star}\left(\sO_{\tX}(((N-1)+k-l)E)\right)\cong \sO_{X_+}$$
for $0\leq (N-1)+k\leq (N-1)$. 
So $\mathbb{GR}_k^{-1}\circ \mathbb{BO}_{(N-1)+k}$ takes $\sO_{X_-}(l)$ to  $\sO_{X_-}(l)$  for $k\leq l\leq k+(N-1)$. The proof that 
$\mathbb{GR}_k^{-1}\circ \mathbb{BO}_{(N-1)+k}$ acts as identity on $\mathcal{E}xt^i( \sO_{X_-}(l) , \sO_{X_-}(l^\prime))$ is the same as 
\cite[Proposition 3.1]{ADM}. 
\end{proof}

Let 
$$j_+: \bP(\mathbf{a})\hookrightarrow X_+; \quad  j_-: \bP(\mathbf{b})\hookrightarrow X_-$$
be the closed immersions for the weighted projective stacks $\bP(\mathbf{a})$ and $\bP(\mathbf{b})$.  To abuse notations, 
we understand $\sO_{\bP(\mathbf{b})}(k)$ as line bundle over $\bP(\mathbf{b})$, and at the same time taken as the coherent sheaf  
$j_{-\star}\sO_{\bP(\mathbf{b})}(k)$ on $X_-$.  The same situation holds for $j_+: \bP(\mathbf{a})\hookrightarrow X_+$.

\begin{proposition}\label{identity:GR}
We have the following result for the autoequivalence:
$$\mathbb{GR}_k^{-1}\circ \mathbb{GR}_{k+1}=\mathbb{T}_{\sO_{\mathbb{P}(\mathbf{b})}(k)}$$
associated with the spherical functor 
$$\mathbb{T}_{\sO_{\mathbb{P}(\mathbf{b})}(k)}: D_{Q}^b(X_-)\to D_{Q}^b(X_-)$$
defined by:
$$\mathbb{T}_{\sO_{\mathbb{P}(\mathbf{b})}(k)}(\mathcal{E})=\Cone(\sO_{\mathbb{P}(\mathbf{b})}(k)\otimes\RHom(\sO_{\mathbb{P}(\mathbf{b})}(k), \mathcal{E})\stackrel{\eval}{\rightarrow}\mathcal{E}).$$
\end{proposition}
\begin{proof}
We generalize the proof  in Proposition 3.2 of \cite{ADM}. 
We prove the $k=0$ case, since other cases are similar. 
It suffices to check that both functors act on  $\sO_{X_-}(1), \cdots, \sO_{X_-}(N)$, since these objects split-generate the derived category $D^b_{Q}(X_-)$. 
Clearly $\mathbb{GR}_0^{-1}\circ \mathbb{GR}_{1}$ and $\mathbb{T}_{\sO_{\mathbb{P}(\mathbf{b})}}$ act on $\sO_{X_-}(1), \cdots, \sO_{X_-}(N-1)$ as identities. This is due to the facts that 
the full subcategories $\mathbf{G}_0\subset D^b_{Q}(X_0)$; and $\mathbf{G}_1\subset D^b_{Q}(X_0)$ are split-generated by the objects
$\sO_{X_0}, \cdots, \sO_{X_0}(N-1)$; and $\sO_{X_0}(1), \cdots, \sO_{X_0}(N)$, respectively. 

We check the case $\sO_{X_-}(N)$. 
Consider the Koszul resolution of the substack 
$[\CC^{M<0}\cap U_0/K]\subset X_0$, which is cut out by a transverse section of $\sO_{X_0}(-1)\otimes S_+$:
$$\sO_{X_0}(N)\otimes \det(S_+^*)\to \cdots \to\sO_{X_0}(2)\otimes \wedge^2(S_+^*)\to \sO_{X_0}(1)\otimes S_+^*\to \sO_{X_0}\to \sO_{[\CC^{M<0}\cap U_0/K]}.$$
Restrict to $X_-$ we get:
\begin{equation}\label{KoszulX-}
\sO_{X_-}(N)\otimes \det(S_+^*)\to \cdots \to\sO_{X_-}(2)\otimes \wedge^2(S_+^*)\to \sO_{X_-}(1)\otimes S_+^*\to \sO_{X_-}\to \sO_{j_{-\star}\sO_{\bP(\mathbf{b})}}.
\end{equation}  
Then we restrict to $X_+$, we get: (Note that there is no last term.)
\begin{equation}\label{KoszulX+}
\sO_{X_+}(-N)\otimes \det(S_+^*)\to \cdots \to\sO_{X_+}(-2)\otimes \wedge^2(S_+^*)\to \sO_{X_+}(-1)\otimes S_+^*\to \sO_{X_+}.
\end{equation}
Now we have 
$$\mathbb{GR}_1(\sO_{X_-}(N))=\sO_{X_+}(-N).$$
Use (\ref{KoszulX+}) we get:
$$\underbrace{\sO_{X_+}(-(N-1))\otimes S_+}_{\deg 0}\to\sO_{X_+}(-(N-2))\otimes \wedge^2(S_+)\to\cdots \to \sO_{X_+}\otimes \det(S_+)$$
Then applying the functor $\mathbb{GR}_0$, 
$$\underbrace{\sO_{X_-}(N-1)\otimes S_+}_{\deg 0}\to\sO_{X_-}(N-2)\otimes \wedge^2(S_+)\to\cdots \to \sO_{X_-}\otimes \det(S_+)$$
which is the middle $N$-terms of (\ref{KoszulX-}) tensored with $\det(S_+)$, and this extension is
$$\Cone(j_{-\star}\sO_{\bP(\mathbf{b})}\otimes\det(S_+)[-N]\stackrel{}{\rightarrow}\sO_{X_-}(N)).$$

On the other hand, the spherical twist 
$$\mathbb{T}_{\sO_{\mathbb{P}(\mathbf{b})}}(\sO_{X_-}(N))=
\Cone(j_{-\star}\sO_{\mathbb{P}(\mathbf{b})}\otimes\RHom(j_{-\star}\sO_{\mathbb{P}(\mathbf{b})}, \sO_{X_-}(N))\stackrel{}{\rightarrow}\sO_{X_-}(N))$$ 
has the same description. 
These two extensions are the same since the functors $\mathbb{GR}_0^{-1}\circ\mathbb{GR}_1$ and $\mathbb{T}_{\sO_{\mathbb{P}(\mathbf{b})}}$ acts in the same way on the Exts. 
\end{proof}

Let 
$$\mathbb{BO}_k^\prime:  D_{Q}^b(X_+)\to D_{Q}^b(X_-)$$
be the general Bondal-Orlov fucntors  other way around by:
$$\mathbb{BO}_k^\prime:=(f_-)_{\star}(\sO_{\tX}(kE)\otimes (f_+)^\star(--)).$$
The degree zero $\mathbb{BO}^\prime_0$ is 
the Fourier-Mukai transform $\mathbb{FM}^\prime$. 
\begin{corollary}
We have:
$$\mathbb{FM}^\prime\circ \mathbb{FM}=\mathbb{T}^{-1}_{\sO_{\bP(\mathbf{b})}(-1)}\circ\cdots\circ 
\mathbb{T}^{-1}_{\sO_{\bP(\mathbf{b})}(-(N-1))}.$$
\end{corollary}
\begin{proof}
The results in Proposition \ref{GR:BO} and Proposition \ref{identity:GR} imply that 
$$\mathbb{BO}_{-k}^\prime\circ \mathbb{BO}_{(N-1)+k+1}=\mathbb{T}_{\sO_{\mathbb{P}(\mathbf{b})}(k)}.$$
Hence we have:
$$\mathbb{BO}_{-k-1}^\prime\circ \mathbb{BO}_{(N-1)+k}=\mathbb{T}^{-1}_{\sO_{\mathbb{P}(\mathbf{b})}(k)}.$$
By Grothendieck duality, we have that 
$\mathbb{BO}_k^{-1}=\mathbb{BO}^\prime_{(N-1)-k}$. Then the result is a direct calculation. 
\end{proof}
\begin{remark}
By a similar argument we have: 
$$\mathbb{FM}\circ \mathbb{FM}^\prime=\mathbb{T}^{-1}_{\sO_{\bP(\mathbf{a})}(-1)}\circ\cdots\circ 
\mathbb{T}^{-1}_{\sO_{\bP(\mathbf{a})}(-(N-1))}.$$
\end{remark}


\end{document}